\documentclass[11pt]{article} % use larger type; default would be 10pt

\usepackage[utf8]{inputenc} % set input encoding (not needed with XeLaTeX)

\usepackage{geometry} % to change the page dimensions
\usepackage{graphicx} % support the \includegraphics command and options

\usepackage[round]{natbib}

\usepackage{booktabs} % for much better looking tables
\usepackage{array} % for better arrays (eg matrices) in maths
\usepackage{paralist} % very flexible & customisable lists (eg. enumerate/itemize, etc.)
\usepackage{verbatim} % adds environment for commenting out blocks of text & for better verbatim
\usepackage{subfig} % make it possible to include more than one captioned figure/table in a single float
\usepackage{amsmath}
\usepackage{amsfonts}
\usepackage{amssymb}
\usepackage{mathrsfs}
\usepackage{bbm}
\usepackage{amsthm}

\usepackage{multirow}

\usepackage[colorlinks=true, allcolors=blue]{hyperref}

\newtheorem{theorem}{Theorem}[section]
\newtheorem{corollary}[theorem]{Corollary}

\newtheorem{proposition}[theorem]{Proposition}

\theoremstyle{definition}
\newtheorem{definition}[theorem]{Definition}
\theoremstyle{remark}
\newtheorem{rem}[theorem]{Remark}
\theoremstyle{remark}
\newtheorem{example}[theorem]{Example}
\numberwithin{equation}{section}

\newcommand{\F}{\mathcal{F}}
\newcommand{\cC}{\mathcal{C}}
\newcommand{\cS}{\mathcal{S}}
\newcommand{\cP}{\mathcal{P}}
\newcommand{\cI}{\mathcal{I}}

\newcommand{\EEE}{\mathbb{E}}
\newcommand{\PPP}{\mathbb{P}}
\newcommand{\RRR}{\mathbb{R}}

\newcommand{\Ex}{\mathrm{E}}
\newcommand{\Hu}{\mathrm{H}}
\newcommand{\Qu}{\mathrm{Q}}

\newcommand{\mm}{M}

\newcommand{\rA}{\mathrm{A}}
\newcommand{\rB}{\mathrm{B}}

\newcommand{\dd}{\mathrm{d}}
\newcommand{\one}{\mathbbm{1}}
\usepackage{fancyhdr} % This should be set AFTER setting up the page geometry
\usepackage{sectsty}
\allsectionsfont{\sffamily\mdseries\upshape} % (See the fntguide.pdf for font help)
\usepackage[nottoc,notlof,notlot]{tocbibind} % Put the bibliography in the ToC
\usepackage[titles,subfigure]{tocloft} % Alter the style of the Table of Contents

\title{Point forecasting and forecast evaluation with generalized Huber loss}

\author{Robert J. Taggart\\Bureau of Meteorology\\robert.taggart@bom.gov.au}

\begin{document}

\maketitle

\begin{abstract}
\noindent Huber loss, its asymmetric variants and their associated functionals (here named \textit{Huber functionals}) are studied in the context of point forecasting and forecast evaluation. The Huber functional of a distribution is the set of minimizers of the expected (asymmetric) Huber loss, is an intermediary between a quantile and corresponding expectile, and also arises in M-estimation. Each Huber functional is elicitable, generating the precise set of minimizers of an expected score, subject to weak regularity conditions on the class of probability distributions, and has a complete characterization of its consistent scoring functions. Such scoring functions admit a mixture representation as a weighted average of elementary scoring functions. Each elementary score can be interpreted as the relative economic loss of using a particular forecast for a class of investment decisions where profits and losses are capped. The relevance of this theory for comparative assessment of weather forecasts is also discussed.
\vspace{10pt}

\noindent\textbf{Keywords:} Consistent scoring function; Decision theory; Forecast ranking; Economic utility; Elicitability; Expectile; Huber loss; Quantile.
\end{abstract}

%%%%%%%%
\section{Introduction}\label{s:intro}
%%%%%%%%

In many fields of human endeavor, it is desirable to make forecasts for an uncertain future. Hence, forecasts should be probabilistic, presented as probability distributions over possible future outcomes \citep{gneiting2014probabilistic}. Nonetheless, many practical situations require forecasters to issue single-valued point forecasts. In this situation, a directive is required about the specific feature or functional of the predictive distribution that is being sought, or about the loss (or scoring) function that is to be minimized \citep{gneiting2011making, ehm2016quantiles}. Examples of functionals include the mean, median, a quantile or  expectile, with the latter recently attracting interest in risk management \citep{bellini2017risk}. Examples of scoring functions include the squared error scoring function $S(x,y)=(x-y)^2$ and absolute error scoring function $S(x,y)=|x-y|$. In the case that the directive is in the form of a statistical functional, it is critical that any scoring function used is appropriate for the task at hand. Ideally, a point forecast sampled from one's predictive distribution by using the requested functional should also minimize one's expected score. That is, the scoring function should be \textit{consistent} for the functional \citep{gneiting2011making}. It is well-known that the squared error scoring function is consistent for the mean and that the absolute error scoring function is consistent for the median. Within this framework, predictive performance is assessed by computing the mean score over a number of forecast cases.

This paper studies Huber loss, and asymmetric variants of Huber loss (Definition~\ref{def:generalised Huber loss}), as a scoring function of point forecasts, along with its associated functional. The classical Huber loss function \citep{huber1964robust} with positive tuning parameter $a$ is given by
\begin{equation}\label{eq:huber scoring intro}
S(x,y) = 
\begin{cases}
\frac{1}{2}(x-y)^2, & |x-y|\leq a \\
a|x-y|-\frac{1}{2}a^2, & |x-y| > a,
\end{cases}
\end{equation}
where $x$ is a point forecast and $y$ the corresponding realization. It applies a quadratic penalty to small errors and a linear penalty to large errors and is an intermediary between the squared error and absolute error scoring functions. Huber loss is used by the Australian Bureau of Meteorology (BoM) to compare predictive performance of temperature and wind speed forecasts with a view to streamlining forecast production \citep{foley2019evidence}, and is described to weather forecasters in that organization as  ``a compromise between the absolute error and the squared error, in an attempt to use the benefits of both of these.'' The motivation for using Huber loss in this applied context is given in Section~\ref{s:Huber loss and BoM}.

For a given predictive distribution, we call the set of point forecasts that minimize expected Huber loss the \textit{Huber mean} of that distribution. The Huber mean is an intermediary between the median and the mean with some appealing properties. It can be described as the midpoint of the `central interval' of length $2a$ of a distribution, where $a$ is the tuning parameter of the corresponding Huber loss function (see Equation~\ref{eq:HF equiv1} and the accompanying geometric interpretation). The Huber mean, unlike the mean, is not dependent on the behavior of the distribution at its tails. At the same time, it accounts for more behavior in the vicinity of the center of the distribution than the median. It is therefore a robust measure of location for a distribution. More generally, the \textit{Huber functional} gives the minimizers of expected asymmetric Huber loss, and is an intermediary between some $\alpha$-quantile and $\alpha$-expectile, as was also noted by \cite{jones1994expectiles} from the perspective of M-estimation. Basic properties of Huber means and Huber functionals are discussed in Section~\ref{s:functionals}, many of which can be traced to \cite{huber1964robust} in the classical symmetric case. 

In the context of point forecasting, an essential property of a statistical functional is that it is \textit{elicitable}; that is, that the functional generates precisely the set of minimizers of some expected score \citep{lambert2008eliciting}. The Huber functional is shown to be elicitable for classes of probability distributions on $\RRR$ (or subintervals of $\RRR$) under weak regularity assumptions (Theorem~\ref{th:consistency}). The class of scoring functions that are consistent for the Huber functional is also characterized, being parameterized by the set of convex functions (also Theorem~\ref{th:consistency}). Edge cases of this characterization recover the general form of scoring functions that are consistent for quantiles \citep{gneiting2011quantiles, thomson1979eliciting} and expectiles \citep{gneiting2011making}.

Determining which consistent scoring function to use is non-trivial since in practice this choice can influence forecast rankings \citep{murphy1977value, schervish1989general, merkle2013choosing} as illustrated in Section~\ref{ss:ranking}. In the case of quantiles and expectiles, \cite{ehm2016quantiles} gave clarity to this issue by showing that each consistent scoring function for those functionals admits a \textit{mixture representation}; that is, can be expressed as a weighted average of elementary scoring functions. Likewise, each consistent scoring function for the Huber functional can be expressed as a weighted average of elementary scoring functions that are consistent for the Huber functional (Theorem~\ref{th:mixrep}). Again, the analogous results for quantiles and expectiles are recoverable as edge cases of this theorem. The Huber functional and its associated elementary scores arise naturally in optimal decision rules for investment problems with fixed up-front costs, where profits and losses are capped (Section~\ref{ss:economic interpretation}). Such models are intermediaries between the classical simple cost--loss decision model (e.g. \cite{richardson2000skill, wolfers2008prediction}) on the one hand, and investment decision models with no bounds on profits or losses \citep{ehm2016quantiles, bellini2017risk} on the other. The mixture representation, along with the economic interpretation of elementary scoring functions and Murphy diagrams, aids interpreting forecast rankings in empirical situations (\cite{ehm2016quantiles}; Section~\ref{ss:murphy}). Applications include, for example, selecting a consistent scoring function that emphasizes predictive performance at the extremes of a variable's range \citep{taggart2021extremes}.

Conclusions are presented in Section~\ref{s:conclusion} and proofs of the main results are given in the appendix.

%%%%%%%%%%%%%%%%%%%%%%%%%%%%%%%%%%%%%
\section{The use of Huber loss for assessing weather forecast quality}\label{s:Huber loss and BoM}
%%%%%%%%%%%%%%%%%%%%%%%%%%%%%%%%%%%%

The following summarizes the context and initial motivation for using Huber loss to assess predictive performance of forecast systems at the Australian BoM. 

The U.S. National Weather Service and the BoM have an operational model where automated gridded weather forecasts are curated and manually adjusted by meteorologists prior to publication for public use. Advances in numerical weather prediction have led to on-going re-evaluation of the role of human forecasters in meteorological service provision \citep{just2020streamlining, sturrock2020changing}. At the BoM, comparative assessment of the predictive performance of official  forecasts, which are issued by meteorologists, and automated forecasts from the Operational Consensus Forecasts (OCF) system \citep{bom2018upgrades}, was performed initially for forecasts of probability of precipitation. One goal was to advise on the likely impact on forecast quality if automation were adopted. The forecast service was well-defined for these precipitation forecasts and the Brier score was used as a consistent scoring function for ranking the two forecast systems and reporting on the statistical significance of those ranks (\cite{griffiths2017advice}, c.f. Section~\ref{ss:ranking}).

However, for some other variables, such as daily maximum temperature, the BoM's forecast service was not clearly defined. Point forecasts (i.e. $\RRR$-valued forecasts) were requested from forecasters but there was no policy on which point forecast was suitable, either in the form of a directive (e.g. issue the mean of the predictive distribution) or of a scoring (or loss) function to be minimized. This lack of service clarity led a team within the BoM to seek a scoring function that would be suitable for penalizing forecast errors when the forecast user group is the heterogeneous public.

First, scoring functions that penalized over- and under-prediction asymmetrically were not considered. Second, anecdotal evidence suggested that maximum temperature forecasts with small errors (where the forecast $x$ and observation $y$ differed by about $1^\circ$C) were generally viewed favorably by the public whereas those with larger forecast errors (around $4$ or $5^\circ$C) were not. Furthermore, the heavy reputational costs associated with egregious errors of comparable size were similar: a $9^\circ$C error was not much worse than an $8^\circ$C error. Hence, for daily maximum temperature forecasts, two clear preferences could be articulated:
\begin{enumerate}
\item five $1^\circ$C errors are preferred over four perfect forecasts followed by a $4^\circ$C error;
\item a $9^\circ$C error followed by a perfect forecast is to be preferred over an $8^\circ$C error followed by a $4^\circ$C error.
\end{enumerate}
Two commonly used scoring functions, the absolute error and squared error functions, do not satisfy both requirements. Table~\ref{tab:motivation} shows that the mean absolute error scores for these error sequences are consonant with Preference 2 but not with Preference 1, while the opposite holds for mean squared error. However, the mean scores generated by the Huber loss scoring function of Equation~(\ref{eq:huber scoring intro}), with $a=3$, are consonant with both preferences. This is the scoring function that was selected.

{\renewcommand{\arraystretch}{1.25}
\begin{table}[]
\begin{center}
\caption{Mean absolute error (MAE), mean squared error (MSE) and mean Huber loss (MHL) for different error sequences. Here, Huber loss is given by Equation~(\ref{eq:huber scoring intro}) with the choice $a=3$.}
\label{tab:motivation}
\begin{tabular}{|r||r|r|r|}
\hline
sequence of errors $(x-y)$ & MAE & MSE  & MHL   \\ \hline
$(1, 1, 1, 1, 1)$	& 1		& 1		& 0.5   \\ 
$(0, 0, 0, 0, 4)$	& 0.8		& 3.2		& 1.5   \\ \hline
$(9, 0)$				& 4.5		& 40.5	& 11.25  \\ 
$(8,4)$				& 6		& 40		& 13.5 \\ \hline
\end{tabular}
\end{center}
\end{table}
}

Huber loss has subsequently been used by the BoM to score daily maximum and minimum temperature forecasts, and hourly temperature, dewpoint temperature and wind magnitude forecasts, each with an appropriate choice of tuning parameter $a$. The remainder of the paper is devoted to developing the theory of Huber loss, and its asymmetric variants, in the context of point forecasting and forecast evaluation, with some applications of the theory illustrated using official BoM forecasts and OCF forecasts.

%%%%%%%%%%%%%%%%%%%%%%%%%%%%%%%%%%%%%
\section{Quantiles, expectiles and Huber functionals}\label{s:functionals}

To begin, we establish some notation. We work in a setting where forecasts are made for some quantity, where the range of possible outcomes belongs to some interval $I\subseteq\RRR$. Forecasts of the quantity can be in the form of a predictive distribution $F$ on $I$ or of a point forecast $x$ in $I$. The realization (or observation) of the quantity will usually be denoted by $y$.

Let $\F(\RRR)$ denote the class of probability measures on the Borel--Lebesgue sets of $\RRR$ and $\F(I)$ denote the subset of probability measures on $I$. For simplicity, we do not distinguish between a measure $F$ in $\F(\RRR)$ and its associated cumulative density function (CDF) $F$. For $F$ in $\F(I)$, write $Y\sim F$ to indicate that a random variable $Y$ has distribution $F$; that is, $\PPP(Y\leq t)=F(t)$ whenever $t\in I$. Throughout, the notation $\EEE_F$ indicates that the expectation is taken with respect to $Y\sim F$.

The power set of a set $A$ will be denoted $\cP(A)$. For a real-valued quantity $X$, we denote by $X_+$ the quantity $\max(0,X)$. The partial derivative with respect to the $i$th argument of a function $g$ is denoted $\partial_i g$.

Whenever $a, b \in [0,\infty]$, define  the `capping function' $\kappa_{a,b}:\RRR \to \RRR$ by
\[\kappa_{a,b}(x) = \max(\min(x,b),-a) \qquad\forall x\in\RRR.\]
That is, $\kappa_{a,b}(x)$ is $x$ capped below by $-a$ and above by $b$. Note that $x_+ = \kappa_{0,\infty}(x)$.

In many contexts users or issuers of forecasts want a relevant point summary $x$ of a predictive distribution $F$. This can be generated by requesting a specific statistical functional of $F$. Given an interval $I\subseteq\RRR$ and some space $\F$ of probability distributions in $\F(I)$, a \textit{statistical functional} (or simply a \textit{functional}) $\mathrm{T}$ on $\F(I)$ is a mapping $\mathrm{T}:\F(I)\to\cP(I)$ \citep{horowitz2006identification, gneiting2011making}. Two important examples are quantiles and expectiles.

\begin{example}\label{ex:quantiles}
Suppose that  $I\subseteq\RRR$ and $\alpha\in(0,1)$. The \textit{$\alpha$-quantile} functional $\Qu^\alpha:\F(I)\to\cP(I)$ is defined by
\[\Qu^{\alpha}(F) = \{x\in I: \lim_{y\uparrow x}F(y) \leq \alpha \leq F(x)\}\]
whenever $F\in \F(I)$. For any $F$, $\Qu^{\alpha}(F)$ is a closed bounded interval of $I$. The two endpoints only differ when the level set $F^{-1}(\alpha)$ contains more than one point, so typically the functional is single valued. The median functional $\Qu^{1/2}$ arises when $\alpha=1/2$. If $q$ is an $\alpha$-quantile of $F$ and $F$ is continuous at $q$ then $F(q)/(1-F(q)) = \alpha/(1-\alpha)$. Figure~\ref{fig:functionals} illustrates the quantiles $\Qu^{1/2}(F)$ (the median) and $\Qu^{0.7}(F)$, where $F$ is the exponential distribution. The aforementioned property is illustrated in the figure via the vertical dashed line segments, whose lengths are in the ratio $\alpha:(1-\alpha)$.
\end{example}

\begin{example}\label{ex:expectiles}
Given an interval $I\subseteq\RRR$, let $\F_1(I)$ denote the space of probability measures $\F(I)$ with finite first moment. The \textit{$\alpha$-expectile} functional $\Ex^\alpha:\F_1(I)\to\cP(I)$ is defined by
\begin{equation}\label{eq:expectiles}
\Ex^\alpha(F) = \left\{x\in I: \alpha\int_x^\infty(y-x)\,\dd F(y) = (1 - \alpha)\int_{-\infty}^x(x-y)\,\dd F(y)  \right\}
\end{equation}
whenever $F\in \F_1(I)$. It can be shown there is a unique solution $x$ to the defining equation, so expectiles are single-valued. Expectiles were introduced by \cite{newey1987asymmetric} in the context of least squares estimation and have recently attracted interest in financial risk management \citep{bellini2017risk}. Expectiles share properties of both expectations as well as quantiles, and nests the mean functional $\Ex^{1/2}$. Using integration by parts, one can show that $\{x\} = \Ex^\alpha(F)$ if and only if 
\[
 \alpha\int_{[x, \infty)\cap I} (1-F(t))\,\dd t = (1-\alpha)\int_{(-\infty,x]\cap I} F(t)\,\dd t.
\]
The latter equation gives a geometric interpretation of the $\alpha$-expectile of $F$. It is the unique point $x$ such that the $(1-\alpha$)-weighted area of the region bounded by $F$ and $0$ on the interval $(-\infty,x]\cap I$ is equal to the $\alpha$-weighted area of the region bounded by $F$ and $1$ on the interval $[x,\infty)\cap I$. Figure~\ref{fig:functionals} illustrates this interpretation, via the areas of the shaded regions, for the expectiles $\Ex^{1/2}(F)$ (i.e. mean) and $\Ex^{0.7}(F)$, where $F$ is the exponential distribution.
\end{example}

\begin{figure}[bt]
\centering
\includegraphics{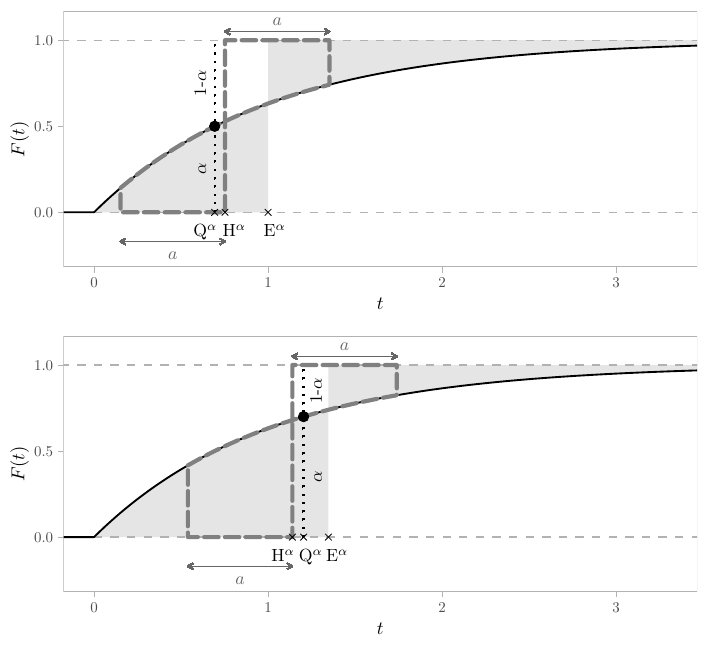}
\caption{The quantile $\Qu^\alpha$, expectile $\Ex^\alpha$ and Huber quantile $\Hu^\alpha$ (where $\Hu^\alpha=\Hu^\alpha_a(F)$, $a=0.6$) when $\alpha=0.5$ (top) and $\alpha=0.7$ (bottom) for the exponential distribution $F(t)=1-\exp(-t),\, t\geq0$. The ratios of the areas of the two shaded regions, of the areas of the two regions bounded by thick dashed lines, and of the lengths of the two dotted line segments, are $\alpha:(1-\alpha)$.}
\label{fig:functionals}
\end{figure}

Equation (\ref{eq:expectiles}) can be re-written as
\begin{equation}\label{eq:expectiles1}
\Ex^\alpha(F) = \left\{x\in I: \alpha\EEE_F\,\kappa_{0,\infty}(Y-x) = (1-\alpha)\EEE_F\,\kappa_{0,\infty}(x-Y)  \right\}.
\end{equation}
By modifying the parameters of the capping function $\kappa_{0,\infty}$, we introduce another functional.

\begin{definition}\label{def:Huber functional}
Suppose that $a > 0$, $b > 0$, $\alpha\in(0,1)$ and that $I\subseteq\RRR$ is an interval. Then the \textit{Huber functional} $\Hu_{a,b}^\alpha:\F(I)\to \cP(I)$ is defined by
\begin{equation}\label{eq:HF def new}
\Hu_{a,b}^\alpha(F) = \left\{ x\in I: \alpha\EEE_F\,\kappa_{0,a}(Y-x) = (1 - \alpha)\EEE_F\,\kappa_{0,b}(x-Y)  \right\}
\end{equation}
whenever $F\in\F(I)$. In the case when $a=b$, we simplify notation and write $\Hu^\alpha_a(F)$ for $H_{a,a}^\alpha(F)$. The special case $\Hu^{1/2}_a(F)$ is called a \textit{Huber mean}.
\end{definition}

We have named the Huber functional for Peter Huber, whose loss function 
\begin{equation}\label{eq:h}
h_a(u) =
\begin{cases}
\frac{1}{2}u^2, & |u|\leq a \\
a|u|-\frac{1}{2}a^2, & |u| > a
\end{cases}
\end{equation}
\citep{huber1964robust} also bears his name. The connection between the Huber functional and Huber loss will be made explicit in Section~\ref{s:scoring}.
Since the Huber functional is an example of a generalized quantile \citep{breckling1988m,jones1994expectiles, bellini2014generalized}, $\Hu^\alpha_{a,b}(F)$ may also be called a \textit{Huber quantile} of $F$. We note here that $x\in\Hu^\alpha_{a,b}(F)$ if and only if $\EEE_FV(x,Y)=0$, where $V: I\times I\to\RRR$ is given by
\begin{equation}\label{eq:identification function}
V(x,y)=|\one_{\{x\geq y\}}-\alpha|\kappa_{a,b}(x-y).
\end{equation}
The function $V$ is an \textit{identification function}  \citep[Section 2.4]{gneiting2011making} for $H^\alpha_{a,b}$, and will be used to establish important properties of the Huber functional.

\begin{figure}[bt]
\centering
\includegraphics{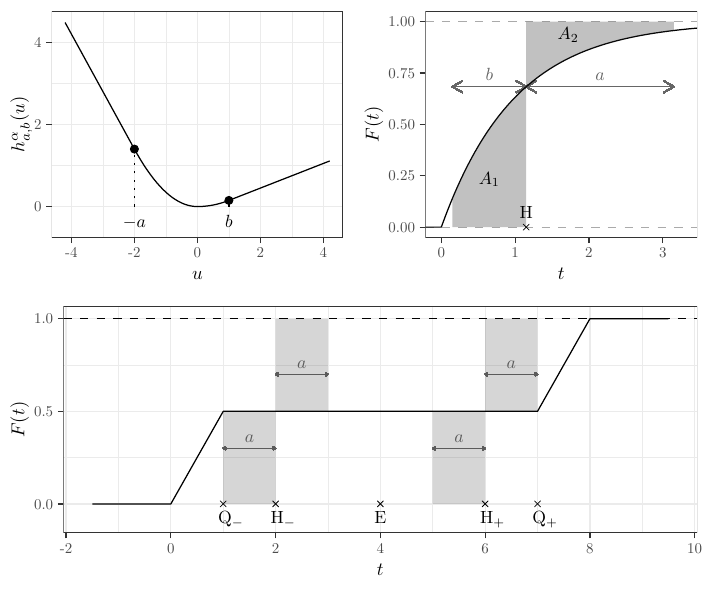}
\caption{Top left: Generalized Huber loss function $h^\alpha_{a,b}$ where $\alpha=0.7$, $a=2$ and $b=1$. Top right: The Huber quantile $\Hu=\Hu^\alpha_{a,b}(F)$ where $\alpha=0.7$, $a=2$ and $b=1$ for the exponential distribution $F(t)=1-\exp(-t),\, t\geq0$. The two shaded areas satisfy the equation $(1-\alpha)A_1=\alpha A_2$. Bottom: A piecewise linear distribution $F$ with endpoints $\Hu_-$ and $\Hu_+$ of the interval $\Hu^{1/2}_a(F)$ where $a=1$, endpoints $\Qu_-$ and $\Qu_+$ of the median interval $\Qu^{1/2}(F)$, and the mean value $\Ex$. The area of each shaded rectangle is equal.}
\label{fig:huberfunctionals}
\end{figure}

As with expectiles, a routine calculation using integration by parts shows that $x\in \Hu_{a,b}^\alpha(F)$ if and only if
\begin{equation}\label{eq:HF equiv1}
\alpha \int_{[x,x+a]\cap I} (1 - F(t))\,\dd t =  (1 - \alpha) \int_{[x-b,x]\cap I} F(t) \, \dd t.
\end{equation}
This gives a geometric interpretation of the Huber functional as the set of points $x$ where the $(1-\alpha)$-weighted area of the region bounded by $F$ and $0$ on $[x-b,x]\cap I$ equals the $\alpha$-weighted area of the region bounded by $F$ and $1$ on $[x,x+a]\cap I$. In the case when $\alpha=1/2$, the two areas are equal. This is illustrated for the exponential distribution in Figure~\ref{fig:functionals} for $\Hu^\alpha_{0.6}(F)$ (when $\alpha=1/2$ and $\alpha=0.7$) and in Figure~\ref{fig:huberfunctionals} for $\Hu^\alpha_{a,b}(F)$ (when $\alpha=0.7$, $a=2$ and $b=1$).

In light of the corresponding geometric interpretations of quantiles and expectiles, and also the similarity between Equations (\ref{eq:expectiles1}) and (\ref{eq:HF def new}), it should come as no surprise that $\alpha$-quantiles and $\alpha$-expectiles are nested as edge cases in the family $\{\Hu^\alpha_a\}_{a\in(0,\infty)}$ of Huber functionals. The following proposition makes this precise and lists several other basic properties of the Huber functional. In what follows, $\overline{F^{-1}(w)}$ denotes the closure in $\RRR$ of the level set $F^{-1}(w)$, and $R(F)$ denotes the smallest closed interval of $\RRR$ that contains the support of the measure $F$.

\begin{proposition}\label{prop:Huber funct properties}
Suppose that $a > 0$, $b > 0$, $\alpha\in(0,1)$, $I\subseteq\RRR$ is an interval and $F\in\F(I)$. 
\begin{enumerate}
\item Then $\Hu_{a,b}^\alpha(F)$ is a nonempty closed bounded subinterval of $I$ contained in $R(F)$.
\item If $\Hu_{a,b}^\alpha(F) = [c,d]$ for some $c<d$, then there exists $w$ in $(0,1)$ such that $\overline{F^{-1}(w)} = [c-b, d+a]$ and $\alpha = bw/(bw+a(1-w))$.
\item If there exists $w$ in $(0,1)$ such that $\overline{F^{-1}(w)} = [c_0, d_0]$ for some $c_0$ and $d_0$ satisfying $d_0-c_0>a+b$, then $\Hu_{a,b}^\alpha(F) = [c_0+b,d_0-a]$ where $\alpha = bw/(bw+a(1-w))$. 
\item $\lim_{a\downarrow0}\min(H_a^\alpha(F)) = \min(\Qu^\alpha(F))$ and $\lim_{a\downarrow0}\max(H_a^\alpha(F)) = \max(\Qu^\alpha(F))$.
\item If $F$ has finite first moment then
\[\lim_{a\to\infty}\min(H_a^\alpha(F)) = \lim_{a\to\infty}\max(H_a^\alpha(F)) = \Ex^\alpha(F)\,.\]
\item If $\tilde F\in\F(I)$ and $F(t)=\tilde F(t)$ whenever
\[
\min(\Hu^\alpha_{a,b}(F))-b \leq t \leq \max(\Hu^\alpha_{a,b}(F)) + a,
\]
then $\Hu^\alpha_{a,b}(F) = \Hu^\alpha_{a,b}(\tilde F)$.
\end{enumerate}
\end{proposition}

Part (1) is similar to Proposition 1(a) of \cite{bellini2014generalized}, whilst parts (4), (5) and (6) were noted, in the case of finite discrete distributions when $a=b$ and $\alpha=1/2$, by \cite{huber1964robust}. The proof is given in the appendix.

Part (6) can be interpreted as saying that the Huber functional only depends on the values of the CDF $F$ away from its tails. In situations where the tail of a predictive distribution is difficult to model, but a point summary describing its broad center is desired, this property is useful. In particular, the Huber functional is invariant to the modification of $F$ outside the interval $[\min(\Hu^\alpha_{a,b}(F))-b, \max(\Hu^\alpha_{a,b}(F)) + a]$. In contrast, modification of the tails of $F$ will generally change its mean and expectile values, whilst quantile values are invariant to modifications of $F$ anywhere apart from at the quantile.

Parts (2) and (3) specify conditions on $F$ for which $\Hu_{a,b}^\alpha(F)$ is multi-valued. Expectiles are always single-valued whereas quantiles can sometimes  be multi-valued. Multi-valued quantiles can arise when $F$ has a bi-modal probability density function (PDF), taking values from the interval $I$ that separates the two distinct densities that comprise the PDF. Huber functionals $\Hu^\alpha_{a,b}(F)$ provide a single-valued alternative to multi-valued quantiles by choosing $a$ and $b$ such that $a+b$ is at least the length of $I$. A precisely stated corollary is that if each level set of $F$ on $R(F)$ has length not exceeding $a+b$ then $\Hu^\alpha_{a,b}(F)$ is single-valued for every $\alpha$ in $(0,1)$. It also follows that $\Hu^\alpha_{a,b}(F)$ is single-valued whenever the quantile $\Qu^\alpha(F)$ is single-valued.

Figure~\ref{fig:huberfunctionals} illustrates a distribution $F$ for which $\Hu^{1/2}_a(F)$ is multi-valued if and only if $0<a<3$. In this particular case, $F$ has a symmetric bi-modal PDF, and also the property that $\Ex^{1/2}(F)\subseteq \Hu^{1/2}_a(F)\subset \Qu^{1/2}(F)$ whenever $a>0$.

Note that while $\Hu^\alpha_a(F)$ is in some sense an intermediary between $\Qu^\alpha(F)$ and $\Ex^\alpha(F)$, the right-hand side of Figure~\ref{fig:functionals} illustrates that the Huber quantile does not always lie between the corresponding quantile and expectile.

%%%%%%%%%%%%%%%%%%%%
\section{Scoring functions, consistency and elicitability}\label{s:scoring}

In this section we discuss scoring functions and their relationship to point forecasts and functionals. Two key concepts are those of consistency and elicitability. How these concepts relate to the Huber functional is the subject of Theorem~\ref{th:consistency}, which is the first major result of this paper.

%%%%%%%%%%%%%%%%%%%%%%%%
\subsection{Scoring functions and Bayes' rules}\label{ss:bayes}

\begin{definition}\label{def:scoring fn}
Suppose that $I\subseteq \RRR$. A function $S: I\times I \to\RRR$ is a called a \textit{scoring function} if $S(x,y)\geq0$ for all $(x,y)\in  I\times I$ with $S(x,y)=0$ whenever $x=y$. The scoring function $S$ is said to be \textit{regular} if (i) for each $x\in I$ the function $y \mapsto S(x,y)$ is measurable, and (ii) for each $y\in I$ the function $x\mapsto S(x,y)$ is continuous, with continuous derivative whenever $x\neq y$.
\end{definition}

The score $S(x,y)$ can be interpreted as the loss or cost accrued when the point forecast $x$ is issued and the observation $y$ realizes. Examples of scoring functions include the squared error scoring function $S(x,y)=(x-y)^2$, the absolute error scoring function $S(x,y)=|x-y|$ and the zero--one scoring function $S(x,y) = \one_{\{|x-y|\geq k\}}(x)$, for some positive $k$. Only first two of these are regular, whilst the zero--one scoring function fails to be regular on account of its discontinuity when $|x-y|=k$. The measurability condition (i) is a technical condition that is satisfied by most (if not all) scoring functions that arise in practice.

Huber loss (\ref{eq:h}) gives rise to the regular scoring function $S(x,y)=h_a(x-y)$. We introduce a more general version.

\begin{definition}\label{def:generalised Huber loss}
Suppose that $a > 0$, $b > 0$ and $\alpha\in(0,1)$. The \textit{generalized Huber loss function} $h_{a,b}^\alpha: \RRR \to \RRR$ is defined by
\[
h_{a,b}^\alpha(u) = 
\begin{cases}
|\one_{\{u\geq0\}}-\alpha|\,\tfrac{1}{2} u^2, & -a \leq u \leq b \\
(1-\alpha)\,b (u - \tfrac{1}{2} b), & u > b \\
-\alpha\, a (u + \tfrac{1}{2} a), & u < -a.
\end{cases}
\]
\end{definition}

The classical Huber loss function given by Equation (\ref{eq:h}) is 2$h^{1/2}_{a,a}$. The same generalization is used by \cite{zhao2021robust} for robust expectile regression. Figure~\ref{fig:huberfunctionals} shows the graph of $h_{2,1}^{0.7}$. Note that $h^\alpha_{a,b}$ is differentiable on $\RRR$, with derivative
\begin{equation}\label{eq:h'}
(h^\alpha_{a,b})'(u) = |\one_{\{u\geq0\}}-\alpha|\kappa_{a,b}(u), \qquad u\in\RRR\,.
\end{equation}
Generalized Huber loss gives rise to the regular scoring function $S(x,y)=h_{a,b}^\alpha(x-y)$.

Given a scoring function $S$, a forecast system that generates point forecasts can be assessed by computing its mean score $\bar S$, where
\[\bar S = \frac1n\sum_{i=1}^nS(x_i,y_i),\]
over a finite set of forecast cases $\{x_1,\ldots,x_n\}$ with corresponding observations  $\{y_1,\ldots,y_n\}$. In this framework, if a number of competing forecast systems are being compared then the one with the lowest mean score is the best performer. Thus, given a scoring function $S$ and predictive distribution $F$, an optimal point forecast is any $\hat x$ in $I$ that minimizes the expected score; that is,
\[\hat x = \arg \min_{x\in I} \EEE_F S(x,Y),\]
provided that the expectation exists. A point forecast that is optimal in this sense is also known as a \textit{Bayes' rule} \citep{gneiting2011making, ferguson1967probability}.

It has long been known that the Bayes' rule under the squared error scoring function $S(x,y)=(x-y)^2$ is the mean of $F$, and under the absolute error scoring function $S(x,y)=|x-y|$ is any median of $F$. The Bayes' rule under the asymmetric piecewise linear scoring function 
\begin{equation}\label{eq:asymmetric piecewise linear}
S(x,y)=|\one_{\{x\geq y\}}-\alpha||x-y|
\end{equation}
is a quantile $\Qu^\alpha(F)$ (e.g. \cite{ferguson1967probability}), whilst the Bayes' rule under the asymmetric quadratic scoring function
\begin{equation}\label{eq:asymmetric quadratic loss}
S(x,y)=|\one_{\{x\geq y\}}-\alpha|(x-y)^2
\end{equation}
is the expectile $\Ex^\alpha(F)$ \citep{newey1987asymmetric, gneiting2011making}. 

To find the Bayes' rule under the generalized Huber loss scoring function $S(x,y)=h^\alpha_{a,b}(x-y)$, we look for solutions $x$ to the equation $\partial_1 \EEE_F S(x,Y) = 0$. If interchanging differentiation and integration can be justified then $\EEE_F \partial_1 S(x,Y) = 0$.
Using Equation (\ref{eq:h'}), one obtains $\EEE_FV(x,Y)=0$, where $V$ is the identification function given by (\ref{eq:identification function}). This implies that $x\in H^\alpha_{a,b}(F)$. So, at least formally, the Bayes' rule under generalized Huber loss is the corresponding Huber functional of $F$. A precise statement will be given in the next subsection.

%%%%%%%%%%%%%%%%%%%%%
\subsection{Consistency and elicitability}\label{ss:consistency}

Whenever a point forecast request specifies what functional of the predictive distribution is being sought, the scoring function used to evaluate the point forecast should be appropriate for that functional.

\begin{definition}\label{def:consistent scoring fn}
\citep{gneiting2011making, murphy1985forecast}
Suppose that $I\subseteq\RRR$. A scoring function $S: I\times I\to \RRR$ is said to be \textit{consistent} for the functional $\mathrm{T}$ relative to a class $\F$ of probability distributions on $I$ if
\begin{equation}\label{eq:consistency def}
\EEE_F S(t,Y) \leq \EEE_F S(x,Y)
\end{equation}
for all probability distributions $F$ in $\F$, all $t$ in $\mathrm{T}(F)$ and all $x$ in $I$. The functional $\mathrm{T}$ is said to be \textit{strictly consistent} relative to the class $\F$ if it is  consistent relative to the class $\F$ and if equality in (\ref{eq:consistency def}) implies that $x \in \mathrm{T}(F)$.
\end{definition}

Evaluating point forecasts with a strictly consistent scoring function rewards forecasters who give truthful point forecast quotes from carefully considered predictive distributions. This is because the requested functional of the predictive distribution coincides with the optimal point forecast (or Bayes' rule).

The families of consistent scoring functions for quantiles and expectiles each have a standard form. Subject to slight regularity conditions, a scoring function $S$ is consistent for the quantile functional $\Qu^\alpha$ if and only if $S$ is of the form
\begin{equation}\label{eq:quantile consistent scoring function}
S(x,y) = |\one_{\{x\geq y\}}-\alpha||g(x)-g(y)|,
\end{equation}
where $g$ is a non-decreasing function \citep{gneiting2011quantiles, thomson1979eliciting, saerens2000building}. Moreover, if $g$ is strictly increasing then $S$ is strictly consistent. The standard asymmetric piecewise linear scoring function (\ref{eq:asymmetric piecewise linear}) for quantiles (which includes, up to a multiplicative constant, the absolute error scoring function for the median) is recovered from Equation (\ref{eq:quantile consistent scoring function}) with the choice $g(t)=t$.

Subject to standard regularity conditions, a scoring function $S$ is consistent for the expectile functional $\Ex^\alpha$ if and only if $S$ is of the form
\begin{equation}\label{eq:expectile consistent scoring function}
S(x,y) = |\one_{\{x\geq y\}}-\alpha|\big(\phi(y)-\phi(x)+\phi'(x)(x-y)\big),
\end{equation}
where $\phi$ is a convex function with subgradient $\phi'$ \citep{gneiting2011making}. Moreover, if $\phi$ is strictly convex then $S$ is strictly consistent. The standard asymmetric quadratic scoring function (\ref{eq:asymmetric quadratic loss}) for expectiles (including, up to a multiplicative constant, the squared error scoring function for the mean) is recovered from (\ref{eq:expectile consistent scoring function}) by taking $\phi(t)=t^2$.  When $\alpha=1/2$, the function $S$ of (\ref{eq:expectile consistent scoring function}) is known as a \textit{Bregman function}.

We will show that consistent scoring functions for the Huber functional also have a standard form. Before doing so, we introduce a critical concept related to the evaluation of point forecasts.

\begin{definition} \citep{lambert2008eliciting}
A statistical functional $\mathrm{T}$ is said to be \textit{elicitable} relative to a class $\F$ of probability distributions if there exists a scoring function $S$ that is strictly consistent for $\mathrm{T}$ relative to $\F$.
\end{definition}

For example, quantiles are elicitable relative to the class $\F(\RRR)$, while expectiles are elicitable relative to the class of distributions in $\F(\RRR)$ with finite first moment \citep{gneiting2011making}. It is worth noting that some statistical functionals are not elicitable, including the sum of two distinct quantiles and \textit{conditional value-at-risk}, a risk measure used in finance \citep{gneiting2011making}.

We turn now to the Huber functional. The main thrust (subject to appropriate regularity conditions) is that the Huber functional is elicitable, and that $S$ is consistent for $\Hu^\alpha_{a,b}$ if and only if $S$ is of the form
\begin{equation}\label{eq:huber scoring function}
S(x,y) = |\one_{\{x\geq y\}} - \alpha|\big(\phi(y) - \phi(\kappa_{a,b}(x-y)+y) + \kappa_{a,b}(x-y)\phi'(x)\big),
\end{equation}
where $\phi$ is a convex function with subgradient $\phi'$. Moreover, $S$ is strictly consistent if $\phi$ is strictly convex. The generalized Huber loss scoring function $S(x,y)=h^\alpha_{a,b}(x-y)$ arises from Equation (\ref{eq:huber scoring function}) with the choice $\phi(t)=t^2$. The following gives a precise statement.

\begin{theorem}\label{th:consistency}
Suppose that $I\subseteq\RRR$ is an interval and that $a > 0$, $b > 0$ and $\alpha\in(0,1)$.
\begin{enumerate}
\item The Huber functional $\Hu_{a,b}^\alpha$ is elicitable relative to the class of probability measures $\F(I)$ when $I$ is bounded or semi-infinite, and elicitable relative to the class of probability measures $\F(I)$ with finite first moment when $I=\RRR$.
\item Suppose that $\phi:I\to\RRR$ is convex on $I$. Then the function $S:I\times I\to\RRR$, defined by Equation (\ref{eq:huber scoring function}),
is a consistent scoring function for the Huber functional $\Hu_{a,b}^\alpha$ relative to the class  $\F(I)$ of probability measures $F$ for which both $\EEE_F[\phi(Y) - \phi(Y-a)]$ and $\EEE_F[\phi(Y) - \phi(Y+b)]$ exist and are finite. If, additionally, $\phi$ is strictly convex then $S$ is strictly consistent for $\Hu_{a,b}^\alpha$ relative to the same class of probability measures.
\item Suppose that the scoring function $S: I\times I\to \RRR$ is regular. If $S$ is consistent for the Huber functional $\Hu_{a,b}^\alpha$ relative to the class of probability measures in $\F(I)$ with compact support, then $S$ is of the form (\ref{eq:huber scoring function}) for some convex function $\phi:I\to\RRR$. Moreover, if $S$ is strictly consistent then $\phi$ is strictly convex.
\end{enumerate}
\end{theorem}

The proof is given in the appendix. 

The general form (\ref{eq:huber scoring function}) for the consistent scoring functions of the Huber functional yields, as edge cases, the general form for the consistent scoring functions of expectiles and quantiles. To be precise, let $S_a^{\Hu,\phi}$ denote the scoring function $S$ given by (\ref{eq:huber scoring function}) when $a=b$, and let $S^{\Ex,\phi}$ and $S^{\Qu,g}$ denote the consistent scoring functions of Equations (\ref{eq:expectile consistent scoring function}) and (\ref{eq:quantile consistent scoring function}) respectively. The relationship between $S_a^{\Hu,\phi}$ and $S^{\Ex,\phi}$ is straightforward via pointwise limit
\begin{equation}\label{eq:Hu Ex pointwise limit}
\lim_{a\to\infty}S_a^{\Hu,\phi}(x,y) = S^{\Ex,\phi}(x,y).
\end{equation}
For the other end of the spectrum we consider the rescaled consistent scoring function $S_a^{\Hu,\phi}/a$, and obtain the pointwise limit
\begin{align}\label{eq:Hu Qu pointwise limit}
\lim_{a\downarrow0}S_a^{\Hu,\phi}(x,y)/a = S^{\Qu,\phi'}(x,y),
\end{align}
where $\phi'$ is nondecreasing because $\phi$ is convex. Importantly, the relevant regularity conditions ensure that every non-decreasing function $g$ in the representation (\ref{eq:quantile consistent scoring function}) is the subderivative of some suitable convex $\phi$.

The consistent scoring functions for the Huber functional thus show a mixture of the properties of the consistent scoring functions for quantiles and expectiles. Focusing on the functional $\Hu^{1/2}_a$ for positive $a$, the only consistent scoring function (up to a multiplicative constant) on $\RRR\times\RRR$ that only depends on the difference $x-y$ between the forecast and observation is the classical Huber loss scoring function $(x,y) \mapsto h^{1/2}_{a,a}(x-y)$. This is because the only Bregman function (up to a multiplicative constant) that has the same property for $\Ex^{1/2}$ is the squared error scoring function $(x,y)\mapsto(x-y)^2$ \citep{savage1971elicitation}. Hence, apart from multiples of classical Huber loss, other consistent scoring functions for $\Hu^{1/2}_a$ on $\RRR\times\RRR$ penalize under- and over-prediction asymmetrically. One such example is the exponential family
\begin{equation}\label{eq:exp scoring fns}
S_{\lambda;a}(x,y) =
\begin{cases}
\frac{1}{\lambda^2}\big(\exp(\lambda y)-\exp(\lambda x)\big) - \frac{1}{\lambda}\exp(\lambda x)(y-x), & |x-y|\leq a\\
\frac{1}{\lambda^2}\big(\exp(\lambda y)-\exp(\lambda (y+a))\big) + \frac{a}{\lambda}\exp(\lambda x), & x-y> a \\
\frac{1}{\lambda^2}\big(\exp(\lambda y)-\exp(\lambda (y-a))\big) - \frac{a}{\lambda}\exp(\lambda x), & x-y< -a,
\end{cases}
\end{equation}
parameterized by $\lambda\in\RRR$ and obtained from (\ref{eq:huber scoring function}) via $\phi(t)=2\exp(\lambda t)/\lambda^2$. These are analogous to the exponential family of Bregman functions considered by \cite{patton2020comparing}.

%%%%%%%%
\section{Mixture representations and Murphy diagrams}\label{s:mixture}
%%%%%%%%

The main theoretical tool presented in this section is the mixture representation for consistent scoring functions of the Huber functional (Theorem~\ref{th:mixrep}). Mixture representations were introduced for quantiles and expectiles by \cite{ehm2016quantiles} and have several very useful applications, including providing insight into forecast rankings.

%%%%%%%%%
\subsection{Ranking of forecasts}\label{ss:ranking}

Recall from Section~\ref{ss:bayes} that point forecasts from two competing forecast systems  $\rA$ and $\rB$ can be ranked by calculating their mean scores $\bar{S}_n^\rA$ and $\bar{S}_n^\rB$ over a finite number $n$ of forecast cases for some scoring function $S$. If the forecast cases are independent, a statistical test for equal predictive performance can be based on the statistic $t_n$, where
\begin{equation}\label{eq:tn}
t_n = \sqrt{n}\, \frac{\bar{S}_n^\rA - \bar{S}_n^\rB}{\hat{\sigma}_n} \quad\text{and}\quad \hat{\sigma}_n^2 = \frac{1}{n}\sum_{i=1}^n(S(x_i^\rA,y_i)-S(x_i^\rB,y_i))^2
\end{equation}
for forecasts $\{x_i^\rA\}$ and $\{x_i^\rB\}$ and corresponding realizations $\{y_i\}$. Subject to traditional regularity conditions, the statistic $t_n$ is standard normal under the null hypothesis of vanishing expected score differentials. Corresponding $p$-values are computed and if the null hypothesis is rejected then $\rA$ is preferred if $t_n<0$ and $\rB$ is preferred otherwise \citep{diebold1995comparing, gneiting2014probabilistic}. Unfortunately, forecast rankings and the results of hypothesis tests can depend on the choice of consistent scoring function \cite[pp.~506, 515--516]{ehm2016quantiles}, as we now illustrate.

\begin{example}\label{ex:OBSH}
Two forecast systems, OCF and MAN, of the Australian BoM produce point forecasts for the daily maximum temperature at Sydney Observatory Hill. The OCF system is fully automated and generates forecasts from a blend of bias-corrected numerical weather prediction forecasts. The MAN forecast is the official forecast of the BoM and is manually issued by meteorologists who have access to various information sources, including OCF. We consider forecasts for the period July 2018 to June 2020 with a lead time of one day. See Figure~\ref{fig:murphy} for a sample time series of MAN and OCF forecasts with observations.

Suppose that these forecasts are targeting the Huber mean $\Hu^{1/2}_3$, and make the simplifying assumption that successive forecast cases are independent. If the consistent scoring function $S(x,y)=2\,h^{1/2}_{3,3}(x-y)$ is used, then the mean score for MAN is lower than the mean score for OCF, and with a $p$-value of $6.52\times10^{-4}$ the null hypothesis of equal predictive performance is rejected at the 5\% significance level in favor of MAN forecasts. However, if the consistent scoring function $S_{2;3}$ defined by Equation (\ref{eq:exp scoring fns}) is used, then OCF has the lower mean score, albeit with a $p$-value of $0.333$ that does not lead to rejection of the null hypothesis.
\end{example}

%%%%%%%%
\subsection{Mixture representations}

In Section~\ref{ss:consistency} it was seen that the class of consistent scoring functions for each quantile, expectile and Huber functional is very large, being parametrized either by the set of nondecreasing functions or by the set of convex functions. The following results show that this apparent multitude can, in a certain sense, be reduced to a one-parameter family of so-called elementary scoring functions.

In general, the choice of function $\phi'$ in the representations (\ref{eq:expectile consistent scoring function}) and (\ref{eq:huber scoring function}) is not unique. To facilitate precise mathematical statements, a special version of $\phi'$ will be chosen. Let $\cI$ denote the class of all left-continuous non-decreasing functions on $\RRR$, and let $\cC$ denote the class of all convex functions $\phi:\RRR\to\RRR$ with subgradient $\phi'$ in $\cI$. This last condition will be satisfied if $\phi'$ is chosen to be the left-hand derivative of $\phi$. Denote by $\cS^\Qu_\alpha$ the class of scoring functions $S$ of the form (\ref{eq:quantile consistent scoring function}) such that $g\in\cI$, by $\cS^\Ex_\alpha$ the class of scoring functions $S$ of the form (\ref{eq:expectile consistent scoring function}) such that $\phi\in\cC$, and by $\cS^\Hu_{\alpha,a,b}$ the class of scoring functions $S$ of the form (\ref{eq:huber scoring function}) such that $\phi\in\cC$. For most practical purposes,  $\cS^\Qu_\alpha$, $\cS^\Ex_\alpha$ and $\cS^\Hu_{\alpha,a,b}$ can be identified with the class of consistent scoring functions for the respective functional on $\RRR$.

The following important result on the representation of scoring functions that are consistent for the quantile and expectile functionals is due to \cite{ehm2016quantiles}.

\begin{theorem}\label{th:ehm mixrep}\citep[Theorem 1]{ehm2016quantiles}
\begin{enumerate}
\item Every member $S$ of the class $\cS^\Qu_\alpha$ has a representation of the form
\begin{equation}\label{eq:mixrep quantiles}
S(x,y) = \int_{-\infty}^{\infty}S^\Qu_{\alpha,\theta}(x,y)\,\dd \mm(\theta), 	\qquad (x,y)\in\RRR^2\,,
\end{equation}
where 
\begin{equation}\label{eq:esf_qu}
S^\Qu_{\alpha,\theta}(x,y) 
=
\begin{cases}
1-\alpha, & y\leq\theta <x, \\
\alpha, & x\leq\theta <y, \\
0, & \text{otherwise,}
\end{cases}
\end{equation}
and $\mm$ is a non-negative measure. The mixing measure $\mm$ is unique and satisfies $\dd \mm(\theta) = \dd g(\theta)$ whenever $\theta\in\RRR$, where $g$ is the nondecreasing function in the representation (\ref{eq:quantile consistent scoring function}). Furthermore, $\mm(x)-\mm(y)=S(x,y)/(1-\alpha)$.
\item Every member $S$ of the class $\cS^\Ex_\alpha$ has a representation of the form
\begin{equation}\label{eq:mixrep expectiles}
S(x,y) = \int_{-\infty}^{\infty}S^\Ex_{\alpha,\theta}(x,y)\,\dd \mm(\theta), 	\qquad (x,y)\in\RRR^2\,,
\end{equation}
where 
\begin{equation}\label{eq:esf_ex}
S^\Ex_{\alpha,\theta}(x,y) 
=
\begin{cases}
(1-\alpha)|\theta - y|, & y\leq\theta <x, \\
\alpha|\theta - y|, & x\leq\theta <y, \\
0, & \text{otherwise,}
\end{cases}
\end{equation}
and $\mm$ is a non-negative measure. The mixing measure $\mm$ is unique and satisfies $\dd \mm(\theta) = \dd\phi'(\theta)$ whenever $\theta\in\RRR$, where $\phi'$ is the left-hand derivative of the convex function $\phi$ in the representation (\ref{eq:expectile consistent scoring function}). Furthermore, $\mm(x)-\mm(y)=\partial_2S(x,y)/(1-\alpha)$.
\end{enumerate}
\end{theorem}

Both integral representations (\ref{eq:mixrep quantiles}) and (\ref{eq:mixrep expectiles}) hold pointwise. The functions defined by (\ref{eq:esf_qu}) and (\ref{eq:esf_ex}) are called \textit{elementary scoring functions} for the quantile and expectile functionals respectively. Thus Theorem~\ref{th:ehm mixrep} essentially states that each scoring function that is consistent for a quantile or expectile functional can be expressed as a weighted average of corresponding elementary scoring functions. The analogous result for Huber functionals is new and stated below.

\begin{theorem}\label{th:mixrep}
Every member $S$ of the class $\cS^\Hu_{\alpha,a,b}$ has a representation of the form
\begin{equation}\label{eq:mixrep}
S(x,y) = \int_{-\infty}^{\infty}S^\Hu_{\alpha,a,b,\theta}(x,y)\,\dd \mm(\theta), 	\qquad (x,y)\in\RRR^2\,,
\end{equation}
where 
\begin{equation}\label{eq:esf}
S^\Hu_{\alpha,a,b,\theta}(x,y) 
=
\begin{cases}
(1-\alpha)\min(\theta - y, b), & y\leq\theta <x, \\
\alpha\min(y - \theta, a), & x\leq\theta <y, \\
0 &\text{otherwise,}
\end{cases}
\end{equation}
and $\mm$ is a non-negative measure. The mixing measure is unique and satisfies $\dd \mm(\theta) = \dd\phi'(\theta)$ whenever $\theta\in\RRR$, where $\phi'$ is the left-hand derivative of the convex function $\phi$ in the representation (\ref{eq:huber scoring function}). Furthermore, $\mm(x)-\mm(y)=\partial_2S(x,y)/(1-\alpha)$.
\end{theorem}

The proof is given in the appendix and is a simple adaptation of the proof for quantiles and expectiles. 

Each function $S^\Hu_{\alpha,a,b,\theta}$ of Theorem~\ref{th:mixrep} is called an \textit{elementary scoring function} for the Huber functional, and also belongs to $\cS^\Hu_{\alpha,a,b}$, as can be seen via Equation (\ref{eq:huber scoring function}) with the choice $\phi(t)=(t-\theta)_+$ and $\phi'(t)=\one_{\{\theta<t\}}$. The mixture representation of Equation~(\ref{eq:mixrep}) holds pointwise. Moreover, when $a=b$, the mixture representations for the consistent scoring functions of expectiles and quantiles emerge as edge cases of Theorem~\ref{th:mixrep} by taking limits as $a\to\infty$ and as $a\downarrow0$ and using the dominated convergence theorem. Details are given in Remark~\ref{rem:mixrep}.

%%%%%%%%%%%%%%%%
\subsection{Economic interpretation of elementary scoring functions}\label{ss:economic interpretation}

\cite{ehm2016quantiles} showed how the elementary scoring functions $S^\Qu_\alpha$ for quantiles have a natural economic interpretation related to binary betting and the classical simple cost--loss decision model (e.g \cite{richardson2000skill, wolfers2008prediction}). On the other hand, the elementary scoring functions $S^\Ex_\alpha$ for expectiles arise naturally in simple investment decisions where profits attract taxation and losses tax deduction, possibly at different rates \citep{ehm2016quantiles, bellini2017risk}. The elementary scoring functions $S^\Hu_{\alpha,a,b,\theta}$ for the Huber functional also admit an economic interpretation. It is the loss, relative to actions based on a perfect forecast, of an investment decision with fixed costs, possibly differential tax rates for profits versus losses, and where profits and losses are capped. This represents an intermediary position between the interpretation for quantiles (where economic losses, if they occur, are fixed irrespective of how near or far the forecast is to the realization) and that for expectiles (where there is no cap on profits or on losses). To illustrate, we give two examples. The first is an adaptation of the interpretation for the elementary scoring functions of expectiles presented by \cite{ehm2016quantiles}, while the second shows how the Huber functional and its elementary functions can arise in the context of investment decisions based on weather forecasts.

\begin{example}\label{ex:alexandra}
Suppose that Alexandra considers investing a fixed amount $\theta$ in a start-up company in exchange for an unknown future amount $y$ of the company's profits or losses. Additionally, Alexandra takes out an option to set a limit $b$ on losses she could incur but which also imposes a limit $a$ on the profits she could receive. Alexandra will make a profit if and only if $y>\theta$, and so adopts the decision rule to invest if and only if her point forecast $x$ of $y$ exceeds $\theta$. Her pay-off structure is as follows:
\begin{enumerate}
\item If Alexandra refrains from the deal, her pay-off will be 0, independent of the outcome $y$.
\item If Alexandra invests and $y\leq \theta$ realizes then her payout is negative at $-(1-r_L)\min(\theta-y, b)$. Here $\min(\theta-y, b)$ is the monetary loss, bounded by $b$, and the factor $1-r_L$ accounts for Alexandra's reduction in income tax with $r_L\in[0,1)$ representing the deduction rate.
\item If Alexandra invests and $y> \theta$ realizes then her pay-off is positive at $(1-r_G)\min(y-\theta,a)$, where $r_G\in[0,1)$ denotes the tax rate that applies to her profits.
\end{enumerate}

The top matrix in Table~\ref{tab:scoring matrix} shows Alexandra's pay-off under her decision rule. The positively-oriented pay-off matrix can be reformulated as a negatively oriented regret matrix, by considering the difference between the pay-off for an (hypothetical) omniscient investor who has access to a perfect forecast and the pay-off for Alexandra. For example, if $x\leq\theta$ and $y>\theta$ realizes, then the omniscient investor's pay-off is $(1-r_G)\min(y - \theta, b)$ while Alexandra's pay-off is 0, and so Alexandra's regret is $(1-r_G)\min(y - \theta, b)$. The bottom matrix of Table~\ref{tab:scoring matrix} is Alexandra's regret matrix, which up to a multiplication factor is the elementary score $S^\Hu_{\alpha,a,b,\theta}(x,y)$. So to minimize regret, Alexandra should invest if and only if $x>\theta$, where $x = \Hu^\alpha_{a,b}(F)$, $F$ is Alexandra's predictive distribution of the future value of the investment and $\alpha=(1-r_G)/(2-r_L-r_G)$. The point forecast $x = H^{1/2}_a(F)$ arises if profits and losses are capped by the same value and if the rates $r_G$ and $r_L$ are equal.
\end{example}

{\renewcommand{\arraystretch}{1.25}
\begin{table}
	\begin{center}
		\caption{Overview of pay-off structure for Alexandra's decision rule to invest if and only if $x>\theta$.}
		\label{tab:scoring matrix}
		\begin{tabular}{|ccc|}
			\hline
			& $y \leq \theta$ & $y > \theta$ \\
			\hline
			Monetary payoff & & \\
			$x\leq\theta$ & 0 & 0 \\
			$x > \theta$ & $-(1-r_L)\min(\theta-y, b)$ &  $(1-r_G)\min(y - \theta, a)$ \\
			\hline
			Score (regret) & & \\
			$x\leq\theta$ & 0 & $(1-r_G)\min(y - \theta, a)$ \\
			$x > \theta$ & $(1-r_L)\min(\theta-y, b)$ &  0 \\
			\hline
		\end{tabular}
	\end{center}
\end{table}
}

\begin{example}\label{ex:hannah}
Hannah runs a business selling ice creams from a mobile cart at a sports stadium. Historically, there is an approximately linear relationship between the volume of ice cream sales on any given afternoon and the observed daily maximum temperature, so that the profit $p$ from sales is modeled by $p=ky+c$, where $y$ is the observed daily maximum temperature, $k>0$ and $c\in\RRR$. Additionally, $0\leq p \leq a$ for some positive $a$, since total sales are limited by cart capacity, while any unsold units can be sold at a later date. If Hannah chooses to sell ice creams on any given afternoon, she must also pay a fixed cost $f$ (staff wages and stadium fees). If model assumptions are correct, Hannah will make a profit if and only if $ky+c>f$. So she adopts the decision rule to sell ice creams on any given afternoon if and only if her point forecast $x$ of the maximum temperature exceeds the decision threshold $\theta$, where $\theta=(f-c)/k$. Her pay-off structure is as follows.
\begin{enumerate}
\item If Hannah does not sell ice creams then her pay-off is 0.
\item If Hannah sells ice creams and $y>\theta$ then her profit after tax is $(1-r_G)\min(ky+c-f, a-f)$, where $r_G\in[0,1)$ denotes the tax rate. Her profit can be rewritten as $(1-r_G)k\min(y-\theta, (a-f)/k)$.
\item If Hannah sells ice creams and $y<\theta$ then her loss after tax deductions is $(1-r_L)\min(f - (ky+c), f)$, where $r_G\in[0,1)$ denotes the deduction rate, and losses are capped by $f$ since unsold ice creams go back into storage. Her loss can be rewritten as $(1-r_L)k\min(\theta-y, f/k)$.
\end{enumerate}
As with Example~\ref{ex:alexandra}, these outcomes can be converted to a regret matrix, which up to a multiplication factor is the elementary score $S^\Hu_{\alpha,(a-f)/k,f/k,\theta}(x,y)$ where $\alpha = (1-r_G)/(2-r_L-r_G)$. Consequently, her optimal decision rule is to sell ice creams if and only if $x>\theta$, where $\theta=(f-c)/k$, $x\in\Hu^{\alpha}_{(a-f)/k,f/k}(F)$, $F$ is her predictive distribution of the maximum temperature and $\alpha = (1-r_G)/(2-r_L-r_G)$.
\end{example}

The essential features of Example~\ref{ex:hannah} also arise in the context of rainfall storage and water trading. Any profits made by selling harvested water are capped by storage capacity. The predicted volume $v$ of water that is collected from any rainfall event can be modeled by $v=ky+c$, where $y$ is the predicted rainfall at a representative point within the catchment, $c$ is catchment initial loss and $k$ is determined by catchment size and continuing loss.

%%%%%%%%%%%%%%%%%%
\subsection{Forecast dominance, Murphy diagrams and choice of consistent scoring function}\label{ss:murphy}

We return to the problem of forecast rankings with the notion of \textit{forecast dominance} \cite[Section~3.2]{ehm2016quantiles}. We say that forecast system A \textit{dominates} forecast system B for point forecasts targeting a specific Huber functional if the expected score of point forecasts from A is not greater than the expected score of point forecasts from B, for every consistent scoring function. In practice this is impossible to check directly because the family of consistent scoring functions, parameterized by $\phi\in\cC$, is very large. However, by the mixture representation of Theorem~(\ref{th:mixrep}), one need only test for dominance over the family, parametrized by $\theta\in\RRR$, of elementary functions. In empirical situations, this is further reduced to checking forecast dominance for finitely many $\theta$. In what follows, we consider tuples $(x_{i\rA}, x_{i\rB}, y_i)$ consisting of the $i$th point forecast from systems A and B along with the corresponding observation $y_i$.

\begin{corollary}\label{cor:empirical dominance}
Suppose that $\alpha\in(0,1)$, $a>0$ and $b>0$. The forecast system $\rA$ empirically dominates $\rB$ for predictions targeting $\Hu^\alpha_{a,b}$ if
\[
\frac{1}{n}\sum_{i=1}^n S^\Hu_{\alpha, a, b, \theta}(x_{i\rA},y_i) \leq \frac{1}{n}\sum_{i=1}^n S^\Hu_{\alpha, a, b, \theta}(x_{i\rB},y_i)
\]
whenever $\theta\in\bigcup\{x_{i\rA}, x_{i\rB}, y_i, y_i-a, y_i+b: 1\leq i\leq n\}$ and in the left-hand limit as $\theta\uparrow\theta_0$, where $\theta_0\in\bigcup\{x_{i\rA}, x_{i\rB}, : 1\leq i\leq n\}$.
\end{corollary}

To see why, note that the score differential $\theta\mapsto d_i(\theta)$ for the $i$th forecast case is piecewise linear and right-continuous, and is zero unless $\theta$ lies between $x_{i\rA}$ and $x_{i\rB}$. The only possible discontinuities are at $x_{i\rA}$ and $x_{i\rB}$, and the only possible changes of slope are at $y_i$, $y_i-a$ and $y_i+b$.

An empirical check for forecast dominance is aided with the use of a \textit{Murphy diagram} \cite[Section~3.3]{ehm2016quantiles}, which is a plot showing the graph of 
\[\theta \mapsto  \frac{1}{n}\sum_{i=1}^n S^\Hu_{\alpha, a, b, \theta}(x_i,y_i)\]
for each forecast source, computed at each of the points $\theta$ of Corollary~\ref{cor:empirical dominance}. The top left of Figure~\ref{fig:murphy} presents the Murphy diagram for three different forecasts targeting the Huber mean $\Hu^{1/2}_3$ of the daily maximum temperature at Sydney Observatory Hill (July 2018 to June 2020). The MAN and OCF forecasts were discussed in Example~\ref{ex:OBSH}. For any given day, the Climate forecast is the Huber mean $\Hu^{1/2}_3$ of 46 observations, sampled from the previous 15 days and from a 31 day period this time last year centered on the day in question. A lower mean score is better.

\begin{figure}[bt]
  \centering
  \includegraphics{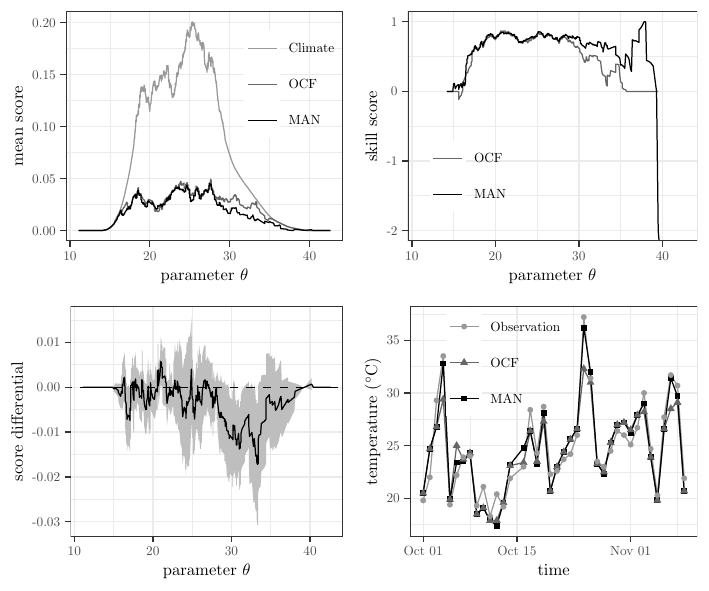}
  \caption{Competing forecast systems targeting the Huber mean $\Hu^{1/2}_3$ for the daily maximum temperature at Sydney Observatory Hill (July 2018 to June 2020). Top left: Murphy diagram of mean elementary scores. Top right: Murphy diagram of elementary skill scores. Bottom left: mean elementary score difference of OCF and MAN with pointwise 95\% confidence intervals (less than 0 indicates that MAN is preferable). Bottom right: Sample of the forecast--observation time series, October and November 2019.}
  \label{fig:murphy}
\end{figure}

The graph in the top right of Figure~\ref{fig:murphy} represents forecast performance as a skill score \citep[Section~2.7]{jolliffe2003forecast} with respect to two reference forecasts: the perfect forecast (skill score = 1) and the Climate forecast (skill score = 0). The difference in mean elementary scores between OCF and MAN forecasts is presented in the bottom left, with pointwise 95\% confidence intervals. Neither of these forecasts dominates the other.

Returning to Example~\ref{ex:hannah}, if Hannah's decision rule is to sell ice creams if and only if the $\Hu^{1/2}_3$ point forecast $x$ exceeds $30^\circ$C, then Hannah should base her decisions on the MAN forecast, since its mean elementary score, which is proportional to economic regret, is lowest (see the top left of Figure~\ref{fig:murphy} where $\theta=30$). But if her fixed investment costs $f$ changed, then so would her decision threshold $\theta$, and the Murphy diagram indicates which forecast system historically performed better at the new threshold.

The initial motivation for the BoM using Huber loss was to inform decisions for streamlining forecast production that are broadly consistent with public appraisal of forecast quality. The decision model associated with the elementary scoring functions of $\Hu^{1/2}_3$ could be taken as a proxy for the decision model of the average forecast user. Under this assumption, the evidence presented in Figure~\ref{fig:murphy} suggests that using the automated OCF forecast would have negligible impact for the average user who has a decision threshold lower than $28^\circ$C. However, evidence points towards human meteorologists (MAN) producing better forecasts for users with temperature decision thresholds from the high 20s to the mid 30s and possibly beyond. Such information can be used to inform BoM policy regarding when its official forecast can be based on the automated system and when meteorologists should intervene. It also indicates where future improvements to the OCF system are possible.

The mixture representation and Murphy diagram also gives insight into why the two different scoring functions of Example~\ref{ex:OBSH} lead to different forecast rankings. The classical Huber loss scoring function $S(x,y)=2\,h^{1/2}_{3,3}(x-y)$ is obtained from Equation~(\ref{eq:huber scoring function}) with the choice $\phi(t)=t^2$. The corresponding mixing measure is $\dd\mm(\theta)=2\,\dd\theta$, implying that every elementary scoring function in the mixture representation (\ref{eq:mixrep}) is weighted equally, and also that the area underneath each graph in the Murphy diagram (top left of Figure~\ref{fig:murphy}) is twice the mean Huber loss $\bar S$ for that forecast system. On the other hand, the exponential scoring function $S_{2;3}$ is obtained from Equation (\ref{eq:huber scoring function}) with the choice $\phi(t)=\exp(2t)/2$. In this case $\dd\mm(\theta)=2\exp(2\theta)\,\dd\theta$ and so mean elementary scores in the corresponding mixture representation are weighted heavily for higher values of $\theta$. Hence when scored by $S_{2;3}$, a slight over-forecast of $40.4^\circ$C by MAN on 19 December 2019 (OCF forecast $35.4^\circ$C and the observation was $39.3^\circ$C) was penalized substantially more heavily than the OCF under-forecast, resulting in a higher mean score $\bar{S}_{2;3}$ for MAN than OCF.

Finally, we consider the choice of consistent scoring function for Huber quantile point predictions in the situation where the point forecast serves the needs of a diverse community of users. Classical Huber loss, obtained when $\phi(t)=2t^2$, applies equal weight to all $\theta$. For everyday use, this choice of $\phi$ may be justified by the desire to weight all decision thresholds $\theta$ equally. On the other hand, for weather forecasts there may be a desire, from a public risk perspective, to give greater weight to values of $\theta$ that lie in the hazardous climatological extremes, so that competing forecast system candidates are evaluated with that in mind. For maximum temperature forecasts, the mixing measure $\dd\mm(\theta) =\phi''(\theta)\,\dd\theta$, where 
\[
\phi''(\theta) = 
\begin{cases}
(5-\theta)+1\,, & \theta \leq 5 \\
1\,, & 5< \theta < 35\\
(\theta-35)+1\,, & \theta\geq 35\,,
\end{cases}
\]
puts increasing weight on decision thresholds below $5^\circ$C and above $35^\circ$C. This yields the convex function
\[
\phi(\theta) = 
\begin{cases}
\frac{1}{6}(5-\theta)^3+\frac{1}{2}\theta^2\,, & \theta\leq 5 \\
\frac{1}{2}\theta^2\,, & 5 < \theta < 35 \\
\frac{1}{6}(\theta-35)^3+\frac{1}{2}\theta^2\,, & \theta\geq 35\,.
\end{cases}
\]
and the corresponding consistent scoring function $S$ can be computed from Equation (\ref{eq:huber scoring function}). With this $S$, MAN maximum temperature forecasts for Sydney outperform those of OCF, and with a $p$-value of $1.48\times10^{-3}$ the null hypothesis of equal predictive performance is rejected at the 5\% significance level.

If comparing predictive performance with emphasis on the extremes is desired, the mixing measure can be designed to concentrate positive weight on the region of interest. For example, choosing $\dd\mm(\theta) = \one_{\{\theta\geq35\}}\,\dd\theta$ emphasizes performance for decision thresholds of at least $35^\circ$C. \cite{taggart2021extremes} discusses evaluation of extremes and decompositions of consistent scoring functions for quantiles, expectiles and Huber means using precisely this approach.

%%%%%%%%%%%%%%%%%%
\section{Conclusion}\label{s:conclusion}

We have defined the Huber functional so that it gives the set of optimal point forecasts for minimizing the expected generalized Huber loss. The Huber functional is an intermediary between quantiles and expectiles, which it nests as edge cases. The Huber functional incorporates more information about a predictive distribution $F$ than quantiles, yet unlike expectiles it is not sensitive to the behavior of $F$ at its tails. We have shown that the Huber functional is elicitable, given a characterization of its consistent scoring functions and stated the mixture representation for those scoring functions. These theoretical results enable the use of the Huber functional and its associated consistent scoring functions within a theoretically sound framework for point forecasting and evaluation (see \cite{gneiting2011making}, \cite{gneiting2014probabilistic}, \cite{ehm2016quantiles} and the references therein).  Moreover, the Huber functional is shown to arise naturally within decision theory for a broad class of investment problems, and within this context the mixture representation facilitates some justification for the choice of consistent scoring function when point forecasts targeting the Huber functional are utilized by a heterogeneous user group. 

Many organizations, including meteorological agencies, have traditionally issued point forecasts that are not well-defined, and that are consumed by a very broad user group. Where there is appetite to clarify forecast definitions, and where it is desirable that point forecasts target some `middle' point of the predictive distribution, the Huber mean provides a good candidate functional, as it combines the local sensitivity of the expectation with the global robustness of the median. The classical Huber loss scoring function $S(x,y)=h^{1/2}_{a,a}(x-y)$ is a natural choice for a consistent scoring function of the Huber mean, as it favors all user-decision thresholds equally (in the sense discussed in Section~\ref{ss:murphy}). Nonetheless, if it is desirable that forecast performance at some user-decision thresholds is more important than at others, the mixture representation provides a method for generating the appropriate scoring function.

\appendix

\section{Proofs}

\begin{proof}[Proof of Proposition~\ref{prop:Huber funct properties}]
We first prove the proposition for the case when $I=\RRR$. 

In light of the essential equivalence of Equations (\ref{eq:HF def new}) and (\ref{eq:HF equiv1}), define $G_{a,b}:\RRR\to\RRR$ by
\begin{equation}\label{eq:G for functional}
G_{a,b}(u) = (1-\alpha)\int_{u-b}^u F(t)\,\dd t - \alpha \int_u^{u+a} (1-F(t))\,\dd t, \qquad u\in \RRR.
\end{equation}
Where then is no confusion, we will drop the subscripts and simply be denote the function by $G$. Since the CDF $F$ is nonnegative and nondecreasing, it follows that $G$ is a nondecreasing function on $\RRR$. Moreover, $G$ is also continuous on $\RRR$.

First we will show that the set of zeroes of $G$ is nonempty and lies in $R(F)$, which will establish that $\Hu_{a,b}^\alpha(F)$ is a nonempty subset of $R(F)$. Since $G$ is continuous and nondecreasing on $\RRR$, it suffices to show that that $G$ takes at least one positive and one negative value in any neighborhood of $R(F)$, which will also establish that the zero set is bounded.

Suppose that $\varepsilon>0$. If $R(F)$ has finite left-endpoint $r_0$ then 
\begin{align*}
G(r_0-\varepsilon) 
&= -\alpha \int_{r_0-\varepsilon}^{r_0-\varepsilon+a} \dd t + \alpha \int_{r_0}^{r_0-\varepsilon+a} F(t)\,\dd t \\
&\leq -\alpha a + \alpha(a+\varepsilon) \\
&< 0 \,.
\end{align*}
Otherwise, let $\eta = \alpha a /((1-\alpha)b +\alpha a)$ and note that $0<\eta<1$. So there exists $v$ in $R(F)$ such that $F(u) <\eta$ whenever $u\leq v$. So
\[
G(v-a) < (1-\alpha)b\eta + \alpha a\eta -\alpha a = 0\,.
\]
Similarly, if $R(F)$ has a finite right-endpoint $r_1$ then $G(r_1+\varepsilon)>0$.
Otherwise, there exists $w$ in $R(F)$ such that $F(u)>\eta$ whenever $u\geq w$. So
\[G(w+b) > (1-\alpha)b\eta + \alpha a\eta -\alpha a = 0\,.\]
This shows that $G$ has at least one zero, and since $\varepsilon$ is arbitrary and $G$ is nondecreasing and continuous, all the zeroes are contained in the interval $R(F)$, and the zero set is a closed bounded interval.

To prove parts (2) and (3)  when $I=\RRR$, we note that the zero set of $G$ is $[c,d]$ where $c<d$ only if there is a constant $w\in(0,1)$ such that $F(t)=w$ whenever $t\in\{\bigcup((u-b,u)\cup(u,u+a)):u\in[c,d]\}$. The closure of this latter set is precisely $[c-b,d+a]$. Moreover, if $u$ is any such zero of $G$, (\ref{eq:G for functional}) implies that
\[0 = G(u) = (1-\alpha)bw +\alpha a w - \alpha w.\]
Rearranging gives  $\alpha = bw/(bw+a(1-w))$ as required.

To prove part (4), fix $\alpha$ and $F$. Define $q_0$ and $q_1$ by
\[q_0 = \min(\Qu^\alpha(F)) \qquad\text{and}\qquad q_1 = \max(\Qu^\alpha(F))\,.\]
Suppose that $a>0$. Denoting $\lim_{y\uparrow x} F(y)$ by $F(x^-)$, note that
\begin{equation}\label{eq:qs}
F(q_i^-)\leq\alpha\leq F(q_i)\,, \quad F(q_0-a) < \alpha \quad\text{and}\quad F(q_1+a) > \alpha\,.
\end{equation}
Therefore
\begin{align*}
G_{a,a}(q_0-a) 
&\leq (1-\alpha) a F(q_0-a) + \alpha a F(q_0) - \alpha a \\
&< a \alpha (F(q_0)-F(q_0-a)) + a\alpha - \alpha a \\
&\leq 0\,,
\end{align*}
and similarly,
\[G_{a,a}(q_0+a) \geq (1-\alpha)\alpha a + \alpha^2 a -\alpha a = 0\,.\]
This shows that $q_0-a < \min(\Hu^\alpha_a(F)) \leq q_0+a$, from which is obtained $\lim_{a\downarrow0}\min(\Hu^\alpha_a(F)) = q_0$.
Similarly, one can show that $G_{a,a}(q_1+a)>0$ and $G_{a,a}(q_1-a)\leq0$, which are used to establish $\lim_{a\downarrow0}\max(\Hu^\alpha_a(F)) = q_1$.

Part (5) follows from the definition of expectiles and the Huber functional.

To prove part (6), let $\tilde G_{a,b}$ denote the function of the form (\ref{eq:G for functional}) defined using $\tilde F$ in the place of $F$, and suppose that $F=\tilde F$ on the interval $[\min(\Hu^\alpha_{a,b}(F))-b, \max(\Hu^\alpha_{a,b}(F)) + a]$. If $x\in\Hu^{\alpha}_{a,b}(F)$ then $G_{a,b}(x)=0$ and hence $\tilde G_{a,b}(x)=0$, whence $\Hu^{\alpha}_{a,b}(F) \subseteq \Hu^{\alpha}_{a,b}(\tilde F)$. The reverse inclusion is obtained similarly.

Finally, for each part, the case when $I\subset\RRR$ can be deduced from the case when $I=\RRR$ by considering the natural extension of $F\in\F(I)$ to $\F(\RRR)$, and using the fact that $H^\alpha_{a,b}(F)\subseteq R(F)\subseteq I$.
\end{proof}

\begin{proof}[Proof of Theorem~\ref{th:consistency}] 
To prove part (2), we take a similar approach to the proof presented in \cite[pp. 38--39]{brehmer2017elicitability} (which follows \cite{gneiting2011making}) of the consistency theorem for expectiles. Suppose that $F\in\F(I)$ and that the expectations $\EEE_F[\phi(Y) - \phi(Y-a)]$ and $\EEE_F[\phi(Y) - \phi(Y+b)]$ both exist and are finite, which will guarantee the existence of the expectations that follow. Fix $a$ and $b$ in $(0,\infty)$ and $\alpha$ in $(0,1)$. For convenience, denote $\kappa_{a,b}$ by $\kappa$. Consider $t\in \Hu^\alpha_{a,b}(F)$ and let $x\in\RRR$. Suppose that $\phi$ is convex and let $S$ be defined by (\ref{eq:huber scoring function}). We need to show that
\begin{equation}\label{eq:ESxt}
\EEE_F S(x,Y) - \EEE_F S(t,Y)\geq0.
\end{equation}
Define the function $g:I \times I\to\RRR$ by
\[g(u,v) = \phi(v) - \phi(u) - \phi'(u)(v-u)\]
and note that $g$ is nonnegative by the convexity of $\phi$, and strictly positive if $\phi$ is strictly convex. Define the function $f:I \times I\to\RRR$ by
\[f(u,v) = \phi'(u) - \phi'(v)\,,\]
and note that $f(u,v)\geq0$ whenever $u\geq v$ by the convexity of $\phi$, with $f(u,v)>0$ whenever $u>v$ if $\phi$ is strictly convex.

To show (\ref{eq:ESxt}), we break it up into two main cases (either $x<t$ or $t<x$) and then into several sub-cases. Consider first the case where $x<t$ with subcase $x-b < x < x+a \leq t-b < t < t+a$. Define the sets $A_i$, where $i\in\{1,2,...,7\}$, by $A_1=\{Y\in(\infty, x-b)\cap I\}$, $A_2=\{Y\in[ x-b,x]\cap I\}$, $A3=\{Y\in(x, x+a]\cap I\}$, $A_4=\{Y\in(x+a, t-b)\cap I\}$, $A_5=\{Y\in[t-b, t]\cap I\}$, $A_6=\{Y\in(t, t+a]\cap I\}$ and $A_7=\{Y\in(t+a, \infty)\cap I\}$. Note that the sets $A_i$ are disjoint and that their union is $I$. Hence
\[
\EEE_F S(x,Y) - \EEE_F S(t,Y) = \EEE_F(S(x,Y)-S(t,Y))\sum_{i=1}^7\one_{A_i}
\]
We will calculate each term in the series and sum them together at the end. The calculations are:
\begin{align*}
&\EEE_F(S(x,Y)-S(t,Y))\one_{A_1} \\
&\quad = (1-\alpha)f(x,t)\EEE_F\kappa(t-Y)\one_{A_1}\,, \\
&\EEE_F(S(x,Y)-S(t,Y))\one_{A_2} \\
&\quad = (1-\alpha)\EEE_F\big(g(x,Y+b) + f(x,t)\kappa(t-Y)\big)\one_{A_2} \,,\\
&\EEE_F(S(x,Y)-S(t,Y))\one_{A_3} \\
&\quad = \EEE_F\big(\alpha S(x,Y)-(1-\alpha)S(t,Y)\big)\one_{A_3} \\
&\quad = \alpha\EEE_F g(x,Y)\one_{A_3} + (1-\alpha)\EEE_F\big(g(Y,Y+b)+bf(Y,x)+f(x,t)\kappa(t-Y)\big)\one_{A_3}\,, \\
&\EEE_F(S(x,Y)-S(t,Y))\one_{A_4} \\
&\quad = \EEE_F\big(\alpha S(x,Y)-(1-\alpha)S(t,Y)\big)\one_{A_4} \\
&\quad = \alpha\EEE_F\big(g(Y-a,Y) + f(Y-a,x)\big)\one_{A_4} \\
&\qquad + (1-\alpha)\EEE_F\big(g(Y,Y+b)+bf(Y,x)+f(x,t)\kappa(t-Y)\big)\one_{A_4} \,,\\
&\EEE_F(S(x,Y)-S(t,Y))\one_{A_5} \\
&\quad = \EEE_F\big(\alpha S(x,Y)-(1-\alpha)S(t,Y)\big)\one_{A_5} \\
&\quad = \alpha\EEE_F\big(g(Y-a,Y) + f(Y-a,x)\big)\one_{A_5} \\
&\qquad + (1-\alpha)\EEE_F\big(g(Y,t)+(t-Y)f(Y,x)+f(x,t)\kappa(t-Y)\big)\one_{A_5} \,,\\
&\EEE_F(S(x,Y)-S(t,Y))\one_{A_6} \\
&\quad = \alpha\EEE_F\big(g(Y-a,t) + (t-Y+a)f(Y-a,x) + f(x,t)\kappa(t-Y)\big)\one_{A_6} \,,\\
&\EEE_F(S(x,Y)-S(t,Y))\one_{A_7} \\
&\quad = \alpha f(x,t)\EEE_F\kappa(t-Y)\one_{A_7}\,.
\end{align*}
Now when summing these terms together, note that since $t\in \Hu^\alpha_{a,b}(F)$, Equation (\ref{eq:HF def new}) implies that
\begin{equation}\label{eq:vanishing sum}
(1-\alpha)\EEE_F\kappa(t-Y)\one_{A_1\cup A_2\cup A_3\cup A_4\cup A_5} + \alpha\EEE_F\kappa(t-Y)\one_{A_6\cup A_7} = 0
\end{equation}
and thus the all terms containing $f(x,t)$ vanish. The remaining terms are all nonnegative by the properties of $f$ and $g$, which establishes (\ref{eq:ESxt}) in this particular subcase and hence that $S$ is consistent for $\Hu^\alpha_{a,b}$.

To prove strict consistency in this subcase, suppose that $\phi$ is strictly convex and that equality holds in (\ref{eq:ESxt}). So we must have
\begin{align*}
0
&= \EEE_F(S(x,Y)-S(t,Y))\sum_{i=1}^7\one_{A_i} \\
&= (1-\alpha)\EEE_F g(x,Y+b)\one_{A_2} + (1-\alpha)\EEE_F g(Y,Y+b)\one_{A_3} + \alpha\EEE_F g(Y-a,Y)\one_{A_4} \\
& \quad + \alpha\EEE_F g(Y-a,Y)\one_{A_5} + \alpha\EEE_F g(Y-a,t)\one_{A_6} + K\,,
\end{align*}
where $K$ can be written as a sum of nonnegative terms, having applied (\ref{eq:vanishing sum}). Each of the terms in the final expression is nonnegative, so for equality to hold they must all equal $0$. Now the terms involving $A_3$, $A_4$ and $A_5$ are all strictly positive unless $\PPP(Y\in A_i)=0$ for $i=3,4,5$. Similarly, the terms involving $A_2$ and $A_6$ are positive unless $\PPP(Y\in A_2\backslash\{x-b\}) = \PPP(Y\in A_6\backslash\{t+a\}) = 0$. Together, this implies that $\PPP(Y\in(x-b,t+a)\cap I)=0$, or equivalently that $F$ is constant on $(x-b,t+a)\cap I$. Combining this with the fact that $t\in \Hu^\alpha_{a,b}(F)$ if and only if (\ref{eq:HF equiv1}) holds, it is easy to see that $x\in \Hu^\alpha_{a,b}(F)$. This establishes strict consistency.

For the main case $x<t$, there are four further subcases:
\begin{align*}
& x-b < x \leq t-b < x+a \leq t < t+a \\
& x-b < t-b < x \leq x+a < t < t+a \\
& x-b < t-b < x < t < x+a < t+a \\
& x-b < x \leq t-b < t \leq x+a < t+a\,.
\end{align*}
The proof of consistency for each subcase proceeds in the same way as the first subcase, and if proceeding in this order most of the calculations in subcases that have already been proved can be used to prove subsequent subcases. The proof of strict consistency also proceeds similarly for the first case, by showing that $F$ is constant on $(x-b,t+a)\cap I$. Details are left to the reader.

The case when $t<x$ is proved the same way, but calculations are quicker by exploiting symmetry and anti-symmetry. For example, the subcase
\[t-b < t < t+a \leq x-b < x < x+a\]
proceeds by switching the roles of $t$ and $x$ in the definitions of $A_i$, and then making the switches $-a \leftrightarrow b$ and $A_i  \leftrightarrow A_{8-i}$ in the calculations for each term. For example, in the case when $t > x$ we have $A_6 = \{Y\in(t,t+a]\cap I\}$ and
\begin{align*}
&\EEE_F(S(x,Y)-S(t,Y))\one_{A_6} \\
&\, = \alpha\EEE_F\big(g(Y-a,t) + (t-Y+a)f(Y-a,x) + f(x,t)\kappa(t-Y)\big)\one_{A_6}\,,
\end{align*}
while in the case when $x < t$, after making switches, we have $A_2=\{Y\in(t-b,t]\cap I\}$ and
\begin{align*}
&\EEE_F(S(x,Y)-S(t,Y))\one_{A_2} \\
&\, = \alpha\EEE_F\big(g(Y+b,t) + (t-Y-b)f(Y+b,x) + f(x,t)\kappa(t-Y)\big)\one_{A_2}\,.
\end{align*}
All the terms are nonnegative apart from those involving $f(x,t)$, which will vanish when all the terms are summed together. Details are left to the reader. This completes the proof of part (2).

To prove part (1) for the cases when $I$ is bounded or semi-finite, use the result of part (2) with the bounded (on $I$) strictly convex function $\phi(t)=e^{-t}$ (or $\phi(t)=e^{t}$ if $I$ is the of the form $(-\infty,c)$). When $I=\RRR$, use the same approach with $\phi(t) = t^2$ and note that $\EEE_F[\phi(Y) - \phi(Y-a)]$ and $\EEE_F[\phi(Y) - \phi(Y+b)]$ exists and is finite if $\EEE_FY$ exists and is finite.

To prove part (3), we apply Osband's principle with the identification function $V$ of Equation (\ref{eq:identification function}). An argument similar to \cite[p. 753, 759]{gneiting2011making} shows that
\[\partial_1S(x,y) = h(x)V(x,y)\]
for $x,y\in I$ and some function $h:I\to I$. Integration by parts yields the representation (\ref{eq:huber scoring function}), where the function $\phi$ is defined by
\[\phi(x) = \int_{x_0}^x\int_{x_0}^v h(u) \, \dd u\,\dd v\]
for some $x_0$ in $I$. Now since $S(x,y)\geq0$ for all $x,y\in I$, it follows from (\ref{eq:huber scoring function}) that $(x-y)\phi'(x) + \phi(y) - \phi(x)\geq0$ whenever $-a \leq x-y \leq b$, which in turn implies that $\phi$ is convex on $I$. If $S$ is strictly consistent, then $S(x,y)>0$ for all non-identical $x$ and $y$ in $I$, whence a similar argument shows that $\phi$ is strictly convex.
\end{proof}

\begin{proof}[Proof of Theorem~\ref{th:mixrep}]
Suppose that $a>0$, $b>0$ and $\phi\in\cC$. Define the function $\Phi:\RRR^2\to\RRR$ by
\[\Phi(x,y) = \phi(y)-\phi(\kappa_{a,b}(x-y)+y) + \kappa_{a,b}(x-y)\phi'(x)\,, \qquad x,y\in\RRR\,.\]
We will show that
\begin{equation}\label{eq:Phi identity}
\Phi(x,y) = 2\int_{-\infty}^\infty S^\Hu_{1/2, a, b, \theta}(x,y)\,\dd\phi'(\theta)\,,
\end{equation}
from whence follows the mixture representation (\ref{eq:mixrep}), the fact that $ \dd\mm(\theta) = \dd\phi'(\theta)$ and the relationship $\mm(x)-\mm(y)=\partial_2S(x,y)/(1-\alpha)$ whenever $x>y$.

To show (\ref{eq:Phi identity}), we break into five cases. For the case $x-y<-a$,
\begin{align*}
\Phi(x,y) 
&= \phi(y)-\phi(y-a)-a\phi'(x) \\
&= a(\phi'(y-a)-\phi'(x)) + (y-\theta)\phi'(\theta)\Big|_{\theta=y-a}^y + \int_{y-a}^y \phi'(\theta)\,\dd\theta \\
&= \int_x^{y-a} a \,\dd\phi'(\theta) + \int_{y-a}^y (y-\theta) \, \dd\phi'(\theta) \\
&= \int_x^y \min(y-\theta,a)\,\dd\phi'(\theta) \\
&= 2\int_{-\infty}^\infty S^\Hu_{1/2, a, b, \theta}(x,y)\,\dd\phi'(\theta)\,.
\end{align*}
The case $x-y>b$ is handled analogously. The case $-a\leq x-y<0$ is essentially the same as the proof of the case $x<y$ for expectiles \cite[p.~529]{ehm2016quantiles}, and the case $0< x-y\leq b$ is analogous. The final case $x=y$ is trivial.

Finally, note that the increments of $\mm$ are determined by $S$ and so the mixing measure is unique.
\end{proof}

\begin{rem}\label{rem:mixrep}
We show how the mixture representations for the consistent scoring functions of quantiles and expectiles (Theorem~\ref{th:ehm mixrep}) emerge as limiting cases of Theorem~\ref{th:mixrep}. Consider the case for expectiles first. For fixed $x$, $y$ and $\theta$ we have
\[S^\Ex_{\alpha,\theta}(x,y) = \lim_{a\to\infty} S^\Hu_{\alpha,a,a,\theta}(x,y)\,.\]
Using the notation and limits following the statement of Theorem~\ref{th:consistency},
\begin{align*}
S^{\Ex,\phi}_\alpha(x,y)
&= \lim_{a\to\infty}S^{\Hu,\phi}_{\alpha,a}(x,y) \\
&= \lim_{a\to\infty}\int_{-\infty}^{\infty}S^\Hu_{\alpha,a,a,\theta}(x,y)\,\dd\mm(\theta)\\
&= \int_{-\infty}^{\infty}S^\Ex_{\alpha,\theta}(x,y)\,\dd\mm(\theta)\,,
\end{align*}
where the interchange of limits and integration in the final equality is justified by the dominated convergence theorem and where $\dd \mm(\theta) = \dd\phi'(\theta)$. This recovers the mixture representation for expectiles. Turning now to quantiles, for fixed $x$, $y$ and $\theta$ we have
\[
\lim_{a\downarrow0} \tfrac{1}{a}S^\Hu_{\alpha,a,a,\theta}(x,y) = 
\begin{cases}
1 - \alpha\,, &\qquad y < \theta <x\\
\alpha\,, &\qquad x \leq \theta <y\\
0 &\qquad \text{otherwise,}
\end{cases}
\]
and so $\lim_{a\downarrow0} \tfrac{1}{a}S^\Hu_{\alpha,a,a,\theta}(x,y) = S^\Qu_{\alpha,\theta}(x,y)$ for almost every $\theta$ (differing only when $\theta=y$). Hence, using the notation and limits following Theorem~\ref{th:consistency} and the dominated convergence theorem,
\begin{align*}
S^{\Qu,\phi'}_\alpha(x,y)
&= \lim_{a\downarrow0}\tfrac{1}{a}S^{\Hu,\phi}_{\alpha,a}(x,y) \\
&= \lim_{a\downarrow0}\int_{-\infty}^{\infty}\tfrac{1}{a}S^\Hu_{\alpha,a,a,\theta}(x,y)\,\dd\mm(\theta)\\
&= \int_{-\infty}^{\infty}S^\Qu_{\alpha,\theta}(x,y)\,\dd\mm(\theta)\,,
\end{align*}
where $\dd \mm(\theta) = \dd\phi'(\theta)$. This recovers the mixture representation for quantiles.
\end{rem}

%%%%%%%%%%%%%%%%%
\section*{Acknowledgements}%\addcontentsline{toc}{section}{Acknowledgements}

The author would like to thank Jonas Brehmer, Professor Tilmann Gneiting, Deryn Griffiths, Robert Fawcett, Nicholas Loveday and two anonymous reviewers for the constructive comments and suggestions, which improved the quality of this manuscript. I also thank my family for their support. Finally, an expression of gratitude to Harry Jack, who introduced me to Huber loss, and to Michael Foley, who four years ago made the decision to start scoring Bureau of Meteorology temperature forecasts using Huber loss. I would not have otherwise embarked on this study.

%\section*{conflict of interest}
%You may be asked to provide a conflict of interest statement during the submission process. Please check the journal's author guidelines for details on what to include in this section. Please ensure you liaise with all co-authors to confirm agreement with the final statement.

% Submissions are not required to reflect the precise reference formatting of the journal (use of italics, bold etc.), however it is important that all key elements of each reference are included.

\end{document}